\documentclass[12pt]
{amsart}
\usepackage{tikz-cd}
\usetikzlibrary{matrix,arrows,decorations.pathmorphing}
\usepackage{amssymb}
\usepackage{amsmath,amscd}
\usepackage{mathrsfs}
\usepackage{amsrefs}
\usepackage{fullpage}
\usepackage{url}
\usepackage{cite}
\newtheorem{thm}{Theorem}[section]
\newtheorem{prop}[thm]{Proposition}
\newtheorem{lem}[thm]{Lemma}
\newtheorem{cor}[thm]{Corollary}
\newtheorem{conj}[thm]{Conjecture}

\theoremstyle{definition}

\newtheorem{rem}[thm]{Remark}

\numberwithin{equation}{section}

\newcommand{\C}{\mathbb{C}}
\newcommand{\p}{\mathbb{P}}

\newcommand{\F}{\mathcal{F}}

\DeclareMathOperator{\Sym}{Sym}

\begin{document}

\title{On a Conjecture on the Variety of Lines on Fano Complete Intersections}

\author{Samir Canning}
\email{srcannin@ucsd.edu}
\maketitle
\begin{abstract}
    The Debarre-de Jong conjecture predicts that the Fano variety of lines on a smooth Fano hypersurface in $\p^n$ is always of the expected dimension. We generalize this conjecture to the case of smooth Fano complete intersections and prove that for a smooth Fano complete intersection $X\subset \p^n$ of hypersurfaces whose degrees sum to at most $7$, the Fano variety of lines on $X$ has the expected dimension. 
\end{abstract}

%%%%%%%%%%%%%%%%%%%%%%%%%%%%%%%%%%%%%%%%%%%%%%%%%%%%%%%%%%%%%%%%%%%
\section{Introduction}
Let $X$ be a \textit{general} degree $d$ hypersurface in $\p^n:=\p^n_{\C}$. It is well known that the Fano variety of lines on $X$, $\F(X)$, has dimension $2n-d-3$. For any smooth degree $d$ hypersurface in $\p^n$, we call this number the \textit{expected dimension} of the Fano variety of lines. The following conjecture of Debarre and de Jong predicts that the Fano variety of lines on a smooth Fano hypersurface always has the expected dimension.
\begin{conj}[Debarre, de Jong]
Let $X\subset \p^n$ be any smooth hypersurface of degree $d$. If $d\leq n$, then $\F(X)$ has the expected dimension.
\end{conj}
The conjecture for $d=3$ is classical. Collino \cite{Collino} proved the result for $d=4$. Debarre \cite{Debarre} proved the $d=5$ case. In \cite{Beheshti}, Beheshti proved the cases $d\leq 6$, and then with different techniques in \cite{Beheshti2}, she proved the cases $d\leq 8$. More recently, in \cite{BeheshtiRiedl}, Beheshti and Riedl proved the conjecture for any $n\geq 2d-4$ . 
\par The purpose of this paper is to state and investigate a generalization of the Debarre-de Jong conjecture for complete intersections. Let $X\subset\p^n$ be a complete intersection of $s$ hypersurfaces of degrees $d_1,\dots,d_s$, and set $d:=\sum d_i$. In Proposition \ref{expected}, we prove a well-known result that for a general such $X$, the dimension of $\F(X)$ is $2n-d-s-2$ if $2n-d-s-2\geq 0$ and $\F(X)$ is empty otherwise. It then seems reasonable to formulate a generalized version of the Debarre-de Jong conjecture.
\begin{conj}\label{conjecture}
Let $X$ be a smooth, Fano complete intersection of $s$ hypersurfaces of degrees $d_1,\dots,d_s$ with $\sum d_i=d$ in $\p^n$. Then $\F(X)$ has the expected dimension $2n-d-s-2$.
\end{conj}
The Fano assumption in Conjecture \ref{conjecture} is necessary. For example, it is well known that the Fano variety of lines on a Fermat hypersurface of degree $d$ in $\p^n$ is of dimension $n-3$, which is larger than expected when $d\geq n$. Moreover, there are $K3$ surfaces containing lines that are complete intersections of Fano hypersurfaces. See, for example, \cite[Chapter 10]{Dolgachev} for the case of an intersection three quadrics in $\p^5$.
\par The main result of this paper proves Conjecture \ref{conjecture} for some low values of $d$ analogously to Beheshti's work in \cite{Beheshti}.
\begin{thm}\label{main}
Let $X$ be a smooth, Fano complete intersection of $s\geq 2$ hypersurfaces of degrees $d_1,\dots,d_s$ in $\p^n$ with $\sum d_i=d$. If $d-s\leq 5$, then $\F(X)$ has the expected dimension $2n-d-s-2$.
\end{thm}
Combined with Beheshti's result, Theorem \ref{main} implies that Conjecture \ref{conjecture} is true for low values of $d$. 
\begin{cor}
Conjecture \ref{conjecture} is true if $d\leq 7$.
\end{cor}
In addition, it resolves the conjecture for some smooth, Fano complete intersections where the sum of the degrees of the hypersurfaces is $d\leq 10$. Let $V(d_1,\dots, d_s)$ denote any smooth, Fano complete intersection of $s$ hypersurfaces with $\sum d_i=d$.  For $8\leq d \leq 10$, Theorem \ref{main} implies that Conjecture \ref{conjecture} is true for $V(2,2,2,2)$, $V(2,2,4)$, $V(2,3,3)$, $V(2,2,2,3)$, and $V(2,2,2,2,2)$.

\noindent\textbf{Acknowledgments}.
I would like to thank my advisor Elham Izadi and David Stapleton for the many helpful conversations. This work was partially supported by NSF grant DMS-1502651.

%%%%%%%%%%%%%%%%%%%%%%%%%%%%%%%%%%%%%%%%%%%%%%%%%%%%%%%%%%%%%%%%%%%

\section{The Expected Dimension}
In this section, we give a proof of a well-known result on the expected dimension of the Fano variety of lines of a smooth complete intersection. 
\begin{prop}\label{expected}
Let $X\subset \p^n$ be a smooth complete intersection of $s$ hypersurfaces of degrees $d_1,\dots,d_s$ with $\sum d_i=d$. Then $\dim \F(X)\geq 2n-d-s-2$. If $X$ is general, then $\dim \F(X)=2n-d-s-2$ if $d\leq 2n-s-2$, and $\F(X)$ is empty otherwise.
\end{prop}
\begin{proof}
Let $V$ be an $(n+1)$-dimensional vector space. Define $\Sym^{\vec{d}}(V^*):=\oplus_{i=1}^{s} \Sym^{d_i}(V^*)$. A form $\vec{f}=(f_1,\dots,f_s)\in \Sym^{\vec{d}}(V^*)$ defines an intersection $X=V(\vec{f})\subset \p^n$ of hypersurfaces of degrees $d_1,\dots, d_s$. Let $G(1,n)$ denote the Grassmannian of projective lines in $\p^n$, and let $I$ be the incidence correspondence defined by
\[
I:=\{(\vec{f},l):l\subset V(\vec{f})\}\subset \Sym^{\vec{d}}(V^*)\times G(1,n)
\]
Let $p_1:I\rightarrow \Sym^{\vec{d}}(V^*)$ and $p_2:I\rightarrow G(1,n)$ be the restrictions of the projection maps to $I$. The fiber $p_1^{-1}(\vec f)$ is naturally identified with $\F(V(\vec{f}))$. For a point $[l]\in G(1,n)$, let $W\subset V$ denote the subspace such that $\p(W)=l$. Then there is a surjection $\Sym^{\vec{d}}(V^*)\rightarrow \Sym^{\vec{d}}(W^*)$ whose kernel corresponds to intersections of hypersurfaces of degrees $d_1,\dots,d_s$ containing $l$. Hence $p_2^{-1}(l)$ is naturally identified with a linear subspace of $\Sym^{\vec{d}}(V^*)$ whose codimension is $\dim\Sym^{\vec{d}}(W^*)=d+s$. It follows that $I$ is smooth and irreducible of codimension $d+s$.
\par Suppose $d>2n-s-2$. Then $\dim I<\dim \Sym^{\vec{d}}(V^*)$, so, for a general choice of $\vec{f}$, $p_1^{-1}(\vec{f})=\F(V(\vec{f}))$ is empty.
\par On the other hand, let $d\leq 2n-s-2$. To prove the proposition, it suffices to show that $p_1$ is surjective. Because we showed $I$ is smooth, it suffices to show that the map induced by $p_1$ on Zariski tangent spaces is surjective at a point. Let $l\in G(1,n)$ be any line. Note that the kernel of the map on tangent spaces $T_{(\vec{f},l)}I\rightarrow T_{\vec{f}}\Sym^{\vec{d}}(V^*)$ is $T_l \F(X)\cong H^0(l,N_{l/X})$. Thus, to show surjectivity of $T_{(\vec{f},l)}I\rightarrow T_{\vec{f}}\Sym^{\vec{d}}(V^*)$, it suffices to show that $h^0(l,N_{l/X})=2n-d-s-2$. We choose coordinates on $\p^n$ so that $l$ is given by $x_2=\dots=x_n=0$. The condition that $l\subset X=V(\vec{f})$ is that for $1\leq i\leq s$, we can write
\[
f_i=x_{2}f_{i2}+\cdots x_{n}f_{in}
\]
We then have the exact sequence
\[
0\rightarrow N_{l/X}\rightarrow \mathcal{O}_{l}(1)^{n-1}\rightarrow \oplus_{i=1}^{s}\mathcal{O}_{l}(d_i)\rightarrow 0
\]
The map $\mathcal{O}_{l}(1)^{n-1}\rightarrow \mathcal{O}_{l}(d_i)$ is given by the matrix $(f_{i2},\dots,f_{in})$. For a general choice of $\vec{f}$, the induced map
\[
H^0(l,\mathcal{O}_l(1)^{n-1})\rightarrow H^0(l,\oplus_{i=1}^{s}\mathcal{O}_{l}(d_i))
\] is given by a matrix of full rank, so it is surjective. By the long exact sequence in cohomology, we have 
\begin{equation*}
\begin{split}
    h^0(l,N_{l/X})&=h^0(l,\mathcal{O}_l(1)^{n-1})-h^0(l,\oplus_{i=1}^{s}\mathcal{O}_l(d_i))\\
    &= 2(n-1)-\sum_{i=1}^{s}(1+d_i)\\
    &=2n-d-s-2.
    \end{split}
\end{equation*}
\end{proof}
\section{Reduction to $d=n$}
In this section, we generalize some lemmas from \cite{Beheshti} to the complete intersection setting. We first show that it suffices to prove Conjecture \ref{conjecture} for $d=n$.
\begin{lem}\label{n=d}
If the dimension of $\F(Y)$ is the expected dimension for every smooth complete intersection $Y$ of $s$ hypersurfaces of degrees $d_1,\dots,d_s$ with $\sum d_i=d$ in $\p^d$, then the dimension of $\F(X)$ is the expected dimension for all smooth complete intersections $X$ of hypersurfaces of degrees $d_1,\dots,d_s$ with $\sum d_i=d$ in $\p^n$ for any $n\geq d$.
\end{lem}
\begin{proof}
Fix an $X$ as in the statement of the lemma. Let $l\subset X$ be a line. By a Bertini type argument, we can choose a linear subspace $\Lambda\cong \p^d$ such that $l\subset \Lambda$ and $\Lambda\cap X=Y$ is smooth. We have that $Y$ is a smooth complete intersection of $s$ hypersurfaces of degrees $(d_1,\dots,d_s)$ in $\Lambda$. Then $\F(Y)$ is the intersection of $\F(X)$ with a $2(n-d)$-dimensional Schubert cycle parameterizing lines in $\p^n$ contained in $\Lambda$. We obtain that
\[
d-s-2=\dim_{[l]}\F(Y)\geq \dim_{[l]}\F(X)-2(n-d).
\]
Therefore, $\dim_{[l]}\F(X)\leq 2n-d-s-2$, as desired.
\end{proof}
\begin{rem}
The proof of this lemma tells us that to prove the conjecture for a given $X$, it suffices to prove the conjecture for a general linear section of $X$. 
\end{rem}
We also need a lemma that gives an upper bound on the dimension of the variety swept out by lines on a Fano complete intersection.
\begin{lem}\label{cover}
Let $X$ be a smooth complete intersection of $s$ hypersurfaces of degrees $d_1,\dots,d_s$ with $\sum d_i=d$ in $\p^n$. If $d\geq n$, then $X$ is not covered by lines.
\end{lem}
\begin{proof}
We suppose that $X$ is covered by lines and prove that in this case, $d\leq n-1$. It suffices to show that the normal bundle to a general line on $X$ is globally generated. Indeed, because $N_{l/X}$ is a vector bundle on $\p^1$, it splits as a direct sum of line bundles,
\[
N_{l/X}=\mathcal{O}(a_1)\oplus\dots\oplus\mathcal{O}(a_{n-s-1}).
\]
Further, each $a_i$ is nonnegative because $N_{l/X}$ is globally generated. From the exact sequence
\[
0\rightarrow N_{l/X}\rightarrow N_{l/\p^n}\rightarrow N_{X/\p^n}|_{l}\rightarrow 0,
\]
we see that $n-1-d=a_1+\dots+a_{n-1-s}\geq 0$. Therefore, $d\leq n-1$.
\par To prove that $N_{l/X}$ is globally generated, define an incidence correspondence $I$ by 
\[
I:=\{(p,l):p\in l\}\subset X\times \F(X).
\]
Let $p_1:I\rightarrow X$ be the restriction of the first projection map, and $p_2:I\rightarrow \F(X)$ the restriction of the second projection. Let $(p,l)$ be a general point of $I$. Because $T_{l}\F(X)=H^0(l,N_{l/X})$, we have the diagram
\[
\begin{CD}
   T_{(p,l)}I  @>dp_{2}>>  H^0(l,N_{l/X})\\
@VV{dp_{1}}V        @VV\beta V\\
    T_{p}X @>\alpha >>  N_{l/X}\otimes\kappa(p)
\end{CD}
\]
The map $\alpha$ comes from the tangent bundle sequence for $l\subset X$.
The map $dp_1$ is surjective because $X$ is covered by lines. Hence, the image of $\alpha\circ dp_1$ is $(n-s-1)$-dimensional. By the commutativity of the diagram, we see that the image of $\beta$ is $(n-s-1)$-dimensional, which implies $N_{l/X}$ is globally generated.
\par 

\end{proof}
\begin{section}{Proof of the Main Theorem}
Our main tools will be the Lefschetz hyperplane theorem, the following result of Beheshti \cite[Theorem 2.1]{Beheshti} and a theorem classifying varieties with many lines, one part of which is due to Segre \cite{Segre} and the other due to Rogora \cite[Theorem 2]{Rogora}.
\begin{thm}[Beheshti]\label{uniruled}
Let $X$ be a smooth complete intersection of $s$ hypersurfaces of degrees $d_1,\dots,d_s$ with $\sum d_i=d$. Let $Y$ be an irreducible subvariety of $\F(X)$ such that the lines on $X$ corresponding to points on $Y$ sweep out a divisor. If $d\geq n-1$, then $Y$ is not uniruled.
\end{thm}
\begin{thm}[Segre, Rogora]\label{Lines}
Let $S$ be a $k$-dimensional subvariety of $\p^n$ such that $n-k\geq 3$. Then $S$ has an at most $2k-2$-parameter family of lines.
\begin{enumerate}
    \item If $S$ has a $(2k-2)$-parameter family of lines, then $S\cong \p^k$.
    \item If $S$ has a $(2k-3)$-parameter family of lines, then $S$ is either a quadric or a $1$-parameter family of $\p^{k-1}$.
    \item If $S$ has a $(2k-4)$-parameter family of lines, then $S$ is either a $1$-parameter family of $(k-1)$-dimensional quadrics, a $2$-parameter family of $\p^{k-2}$, or the intersection of $6-k$ hyperplanes with the Grassmannian $G(1,4)\subset \p^9$ in its Pl\"ucker embedding.
\end{enumerate}
\end{thm}
With these tools, we will now prove Theorem \ref{main}.
\begin{proof}[Proof of Theorem \ref{main}]
By Lemma \ref{n=d}, we can reduce to the case $d=n$. First, we deal with the case $d-s=2$. In this case, $\dim(X)=2$, so by Lemma \ref{cover}, the subvariety swept out by lines on $X$ has dimension at most one. Therefore, there are only finitely many lines on $X$, so $\dim \F(X)=0$ as expected.
\par Next, suppose that $d-s=3$. We have $\dim(X)=3$, so the dimension of the subvariety swept out by lines on $X$ is at most $2$. Any two dimensional variety with a $2$-parameter family of lines is isomorphic to $\p^2$ by Theorem \ref{Lines}, but because $X$ is smooth this is impossible by the Lefschetz hyperplane theorem. Therefore, $\dim \F(X)=1$ as expected.
\par Now suppose that $d-s=4$. Then $\dim(X)=4$, and the variety $\Sigma$ swept out by lines on $X$ is of dimension at most $3$ by Lemma \ref{cover}. Because the expected dimension of $\F(X)$ is $2$, we can assume by Theorem \ref{Lines} that $\dim(\Sigma)=3$, as a $2$-dimensional variety contains at most a $2$-parameter family of lines. Also by Theorem \ref{Lines}, we know that $\Sigma$ can have at most a $4$-parameter family of lines. If it has a $4$-parameter family of lines, then $\Sigma$ is isomorphic to $\p^3$ by Theorem \ref{Lines}, which is impossible by the Lefschetz hyperplane theorem. If it has a $3$-parameter family of lines, then it is either a quadric threefold or a $1$-parameter family of planes. The former contradicts the Lefschetz hyperplane theorem. For the latter case, note that the union of the Fano variety of lines on the planes which sweep out $\Sigma$ form a component of the Fano variety of lines on $\Sigma$. This component is uniruled, which contradicts Theorem \ref{uniruled}.
\par Finally, suppose that $d-s=5$, so that $\dim(X)=5$. By Lemma \ref{cover}, the dimension of the variety $\Sigma$ swept out by lines on $X$ is at most $4$. If $\Sigma$ is a threefold, it can have at most a $4$-parameter family of lines by Theorem \ref{Lines}. If it does have a $4$-parameter family of lines, then $\Sigma\cong \p^3$ by Theorem \ref{Lines}, which is impossible by the Lefschetz hyperplane theorem. So we can assume that $\Sigma$ is a fourfold. It can have at most a $6$-parameter family of lines on it. Again by the Lefschetz hyperplane theorem, $X$ does not contain a $\p^4$, so $\Sigma$ cannot have a $6$-parameter family of lines by Theorem \ref{Lines}. The Lefschetz hyperplane theorem also rules out the possibility that $\Sigma$ has a $5$-parameter family of lines. Indeed, $X$ cannot contain a quadric fourfold, which is the remaining possibility from Theorem \ref{Lines} if $\Sigma$ has a $5$-parameter family of lines, as we have already ruled out the case of $X$ containing a $\p^3$. Finally, we exclude the case where $\Sigma$ has a $4$-parameter family of lines. By Theorem \ref{Lines}, this situation occurs when $\Sigma$ contains a $1$-parameter family of quadric threefolds, a $2$-parameter family of planes, or the intersection of $2$ hyperplanes with the Grassmannian $G(1,4)$ in its Pl\"ucker embedding. The Lefschetz hyperplane theorem excludes the possibility of $X$ containing a quadric threefold or a linear section of the Grassmannian $G(1,4)$ in its Pl\"ucker embedding, which is of degree five. The case in which $\Sigma$ is a $2$-parameter family of planes cannot occur by Theorem \ref{uniruled}. Indeed, for dimension reasons, the union of the Fano variety of lines in the planes which sweep out $\Sigma$ form a component of the Fano variety of lines on $\Sigma$, and hence on $X$. This component is uniruled.
\end{proof}
\end{section}
\bibliographystyle{amsplain}
\bibliography{lines}

\begin{bibdiv}
\begin{biblist}

\bib{Beheshti}{article}{
      author={Beheshti, R.},
       title={Lines on projective hypersurfaces},
        date={2006},
        ISSN={0075-4102},
     journal={J. Reine Angew. Math.},
      volume={592},
       pages={1\ndash 21},
         url={https://doi.org/10.1515/CRELLE.2006.020},
      review={\MR{2222727}},
}

\bib{Beheshti2}{article}{
      author={Beheshti, R.},
       title={Hypersurfaces with too many rational curves},
        date={2014},
        ISSN={0025-5831},
     journal={Math. Ann.},
      volume={360},
      number={3-4},
       pages={753\ndash 768},
         url={https://doi.org/10.1007/s00208-014-1024-8},
      review={\MR{3273644}},
}

\bib{BeheshtiRiedl}{article}{
      author={{Beheshti}, Roya},
      author={{Riedl}, Eric},
       title={{Linear subspaces of hypersurfaces}},
        date={2019Mar},
     journal={arXiv e-prints},
       pages={arXiv:1903.02481},
      eprint={1903.02481},
}

\bib{Collino}{article}{
      author={Collino, A.},
       title={Lines on quartic threefolds},
        date={1979},
        ISSN={0024-6107},
     journal={J. London Math. Soc. (2)},
      volume={19},
      number={2},
       pages={257\ndash 267},
         url={https://doi.org/10.1112/jlms/s2-19.2.257},
      review={\MR{533324}},
}

\bib{Debarre}{unpublished}{
      author={Debarre, O.},
       title={Lines on smooth hypersurfaces},
        date={2003},
        note={Preprint at
  \url{https://www.math.ens.fr/~debarre/Lines_hypersurfaces.pdf}},
}

\bib{Dolgachev}{book}{
      author={Dolgachev, Igor~V.},
       title={Classical algebraic geometry},
   publisher={Cambridge University Press, Cambridge},
        date={2012},
        ISBN={978-1-107-01765-8},
         url={https://doi.org/10.1017/CBO9781139084437},
        note={A modern view},
      review={\MR{2964027}},
}

\bib{Rogora}{article}{
      author={Rogora, Enrico},
       title={Varieties with many lines},
        date={1994},
        ISSN={0025-2611},
     journal={Manuscripta Math.},
      volume={82},
      number={2},
       pages={207\ndash 226},
         url={https://doi.org/10.1007/BF02567698},
      review={\MR{1256160}},
}

\bib{Segre}{article}{
      author={Segre, Beniamino},
       title={Sulle {$V_n$} contenenti pi\`u di {$\infty^{n-k}S_k$}. {I and
  II}},
        date={1948},
        ISSN={0392-7881},
     journal={Atti Accad. Naz. Lincei Rend. Cl. Sci. Fis. Mat. Nat. (8)},
      volume={5},
       pages={193\ndash 197 and 275\ndash 280},
      review={\MR{36042}},
}

\end{biblist}
\end{bibdiv}

\end{document}